\documentclass[a4paper,12pt]{amsart}
\usepackage{amsmath}
\usepackage{amssymb}
\usepackage{amsfonts}
\usepackage{amsthm}
\usepackage{mathtools}
\usepackage{mathrsfs}
\usepackage{bm}
\usepackage{microtype}
\usepackage{enumitem}
\usepackage{graphicx}
\usepackage{xcolor}
\usepackage{booktabs}
\usepackage{array}
\usepackage{multirow}
\usepackage{indentfirst}
\usepackage{mathrsfs,eucal,dsfont,verbatim}
\usepackage{algorithm}
\usepackage{algpseudocode}
\usepackage{diagbox}
\usepackage{float}
\allowdisplaybreaks

\usepackage{hyperref}
\hypersetup{
    colorlinks=true,
    linkcolor=blue,
    citecolor=blue,
    urlcolor=blue
}
\usepackage[a4paper, margin=1in]{geometry}
\usepackage{setspace}
\usepackage{icomma}
\numberwithin{equation}{section}

\DeclareMathOperator{\Var}{Var}
\DeclareMathOperator{\Cov}{Cov}

\newcommand{\E}[1]{\mathbb{E}\left[#1\right]}
\newcommand{\Prob}[1]{\mathbb{P}\left(#1\right)}
\newcommand{\Real}{\mathbb{R}}

\newcommand{\eps}{\varepsilon}
\newcommand{\tN}{\widetilde{N}}
\newcommand{\I}{\mathds 1}
\newcommand{\Ff}{\mathcal{F}}
\newcommand{\Ee}{\mathds{E}}
\newcommand{\Ww}{\mathcal{W}}
\newcommand{\Pp}{\mathcal{P}}

\newtheorem{theorem}{Theorem}[section]
\newtheorem{lemma}[theorem]{Lemma}

\newtheorem{corollary}[theorem]{Corollary}
\newtheorem{assumption}{Assumption}

\theoremstyle{definition}

\theoremstyle{remark}
\newtheorem{remark}{Remark}[section]

\usepackage{comment}
\usepackage{color}

\usepackage[utf8]{inputenc}
\usepackage[backend=biber, style=alphabetic, sorting=nyt]{biblatex}
\usepackage{makecell}
\usepackage{caption}

\begin{document}

\title{Strong Convergence Rates for Euler Schemes of L\'evy-Driven SDE using Dynamic Cutting}

\author[D. Platonov]{Denis Platonov}
\address[D. Platonov]{Kyiv T. Shevchenko University\\ Department of Mechanics and Mathematics\\ Acad. Glushkov Ave., 02000,  Kyiv, Ukraine\\
    \texttt{Email: $dplatonov$@$knu.ua$} }

\author[V. Knopova]{Victoria Knopova}
\address[V. Knopova]{Kyiv T. Shevchenko University\\ Department of Mechanics and Mathematics\\ Acad. Glushkov Ave., 02000,  Kyiv, Ukraine\\ \texttt{Email: $vicknopova$@$knu.ua$}}

\begin{abstract}
We establish strong $L^p$ convergence rates for the Euler--Maruyama scheme of L\'evy-driven SDEs. Our approach introduces a novel dynamic cutting (DC) technique that adaptively truncates the L\'evy measure via a time-dependent threshold $\tau^\pm((sh)^\eps)$ linked to the discretization step $h\sim 1/n$. This dynamic threshold contrasts with the fixed one used in the Asmussen–Rosi\'nski (AR) method. Based on this separation, our scheme simulates large jumps via an inhomogeneous compound Poisson process, while small jumps are either omitted or approximated by a Gaussian term. We derive rigorous strong $L^p$ error bounds under suitable conditions on the SDE coefficients and L\'evy measure. Numerical experiments using truncated stable-like processes validate our theory and demonstrate that dynamic cutting, especially with the Gaussian approximation, achieves superior accuracy compared to the AR method. A key reason is controlling the expected number of simulated large jumps at an $O(n)$ rate versus AR's $O(n^\alpha)$ rate, making the method efficient when the jump activity is high (larger $\alpha$).
\end{abstract}



\date{\today}
\maketitle
\tableofcontents

\section{Introduction }

Consider a L\'evy-driven SDE of the type
\begin{equation}\label{BM_SDE}
    X_t = X_0 + \int_0^t a(s, X_s) \, ds + \int_0^t b(s, X_s) \, dB_s + \int_0^t \!\!\int_{-\infty}^\infty c(s, X_{s-}, z) \tN (ds, dz),
\end{equation}
where $\tN(ds,dz)$ is the compensated  Poisson jump measure corresponding to a L\'evy process $Z_t$, and $a(s,x)$, $b(s,x)$, $c(s,x,z)$ are some measurable functions, conditions on which we specify below.

Protter and Talay \cite{PT97} initiated the study of the Euler scheme for such SDEs by establishing weak convergence rates. In particular, they showed that, under suitable conditions on a function $g$, the weak error $\E{g(X_T)} - \E{g(X_T^n)}$ admits an asymptotic expansion in powers of $n^{-1}$. Later, Jacod \cite{J04} extended this framework by proving limit theorems for Euler approximations and studying the asymptotic behavior of the normalized error process $u_n(X_t - X_t^n)$ for an appropriate scaling sequence $u_n \to \infty$.
 Considerable focus has been placed on SDEs driven by $\alpha$-stable processes. For $\alpha \in [1,2)$, strong convergence rates for the Euler--Maruyama approximation have been derived in \cite{MPT16}, \cite{MX18}, and \cite{KS19}. To extend these results, Li and Zhao \cite{YZ24} derived strong $L^p$-error bounds  (for all $p>0$) for the full range $\alpha \in (0,2)$. In parallel, Butkovsky et al. \cite{But24} investigated the strong convergence of the Euler--Maruyama scheme for multidimensional SDEs with $\beta$-H\"older drift ($\beta>0$) driven by an $\alpha$-stable L\'evy processes with $\alpha \in (0,2]$. Their work established strong $L^p$-error bounds that are independent of $p$ for $\alpha \in (2/3,2]$ and demonstrated the optimality of these rates.

From a practical standpoint, simulation of such SDEs is challenging because the exact distribution of the L\'evy increments is generally unknown except for some special cases. Moreover, L\'evy processes can exhibit infinitely many small jumps in any given finite interval, making direct simulation infeasible. Such a behavior is known as infinite activity:
\begin{equation}\label{inf_act}
    \int_{\Real} \nu(dz) = +\infty.
\end{equation}

For the processes with finite activity, simulation reduces to that of a compound Poisson process; however, infinite activity requires alternative approaches. Asmussen and Rosi\'nski \cite{AR01} addressed this issue by replacing the cumulative effect of small jumps (those below a fixed threshold $\eps$) with a Gaussian term whose variance matches that of the truncated small jumps. Based on this idea, Bossy and Maurer \cite{BM24} introduced the $\eps$-Euler--Maruyama scheme for jump--driven SDEs, deriving both strong and weak convergence rates for $p\ge 2$ and identifying an optimal choice of $\eps$.

The aim of this paper is to derive the strong convergence rates, relying  on the \emph{dynamic cutting} procedure, which we hope to be an accurate and effective tool for approximation of the process $X$. The idea is to approximate a L\'evy process by a compound Poisson type process, obtained from  $Z$  by deleting small jumps together with their compensator in a time--dependent way. More precisely, in contrast to \cite{AR01} where the small jumps of certain size were removed, we delete the jumps whose size depends on time (see Picture~\ref{ar_vs_dc} and \eqref{tau_def} for the definition of functions $\tau^\pm$). This idea was introduced in \cite{KK13}, \cite{K14} for deriving the small-time approximation of the transition probability density of a L\'evy process. In this paper we continue the research started in \cite{IKP25}, where we applied dynamic cutting to find the strong convergence rates for approximation of a L\'evy process and provided some numerical examples.
We prove the strong convergence rates
for the Euler scheme for $X$ (see Theorem~\ref{tstr}) and illustrate our results by numerical experiments. In Table~\ref{Tab1} we provide the differences in errors, obtained in the Asmussen-Rosi\'nski scheme (AR), compared to ours (DC). As one can observe,  the AR method exhibits bigger error, which increases as the jump intensity becomes larger.

\begin{figure}[H]
    \centering
    \includegraphics[width=0.7\textwidth]{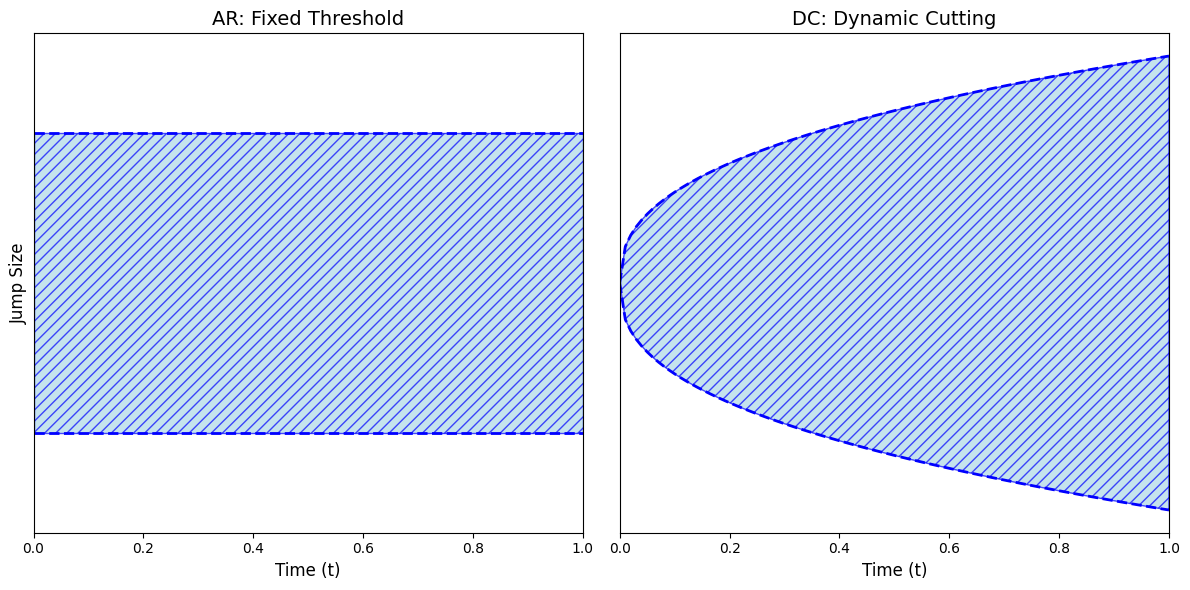}
    \caption{AR vs DC}
    \label{ar_vs_dc}
\end{figure}

\section{Preliminaries and Assumptions}

We write $f\lesssim g$ if there exists a generic constant $C>0$ such that $f \leq C g$. The notation $ f\asymp g$ means that $f\lesssim g$ and $g\lesssim f$.
Since we are interested in the order of convergence, the particular constants are not crucial for us; therefore, we write $\lesssim$ or $\gtrsim$ instead of writing inequalities with explicit constants.

We assume that the L\'evy process $Z$ admits the L\'evy-Ito decomposition
\begin{equation}\label{Set-10}
    Z_t = \int_0^t \!\!\int_\Real z (N(ds,dz) - \nu(dz)ds) =\int_0^t \!\!\int_\Real z \tN(ds,dz),
\end{equation}
where
\[
    N((0, t], B) := \#\{ s \in (0, t] : \Delta Z_s \in B \}, \quad B \in \mathcal{B}(\Real \setminus \{0\})
\]
is the Poisson random measure associated with $Z$, $\nu$ is the respective L\'evy measure, and
$$
\tN(ds,dz) = N(ds,dz)- \nu(dz)ds
$$
is the compensated Poisson random measure.

Define for $r > 0$
\begin{align} \label{Set11}
    N^+(r) := \nu((r, \infty)), &\quad
    N^-(r) := \nu((-\infty, -r)), \nonumber \\
    N(r) &:= N^+(r) + N^-(r).
\end{align}

Denote  the generalized inverses of $ N^\pm $ at $ t^{-1} $ for $t > 0$ by
\begin{align}
    \tau^\pm(t) &:= \sup \left\{ r \geq 0 : N^\pm(r) \geq \frac{1}{t} \right\}, \label{tau_def} \\
    \tau(t) &:= \max \left( \tau^+(t), \tau^-(t) \right). \label{tau}
\end{align}

Define the \emph{Pruitt functions}, see \cite{Pr81}:
\begin{equation}\label{Set15}
    \psi^{L,\pm}(\xi) := \int_{u\neq 0,|u \xi|\leq 1}|u \xi|^2 \, \nu^\pm(du), \quad
    \psi^L(\xi) := \psi^{L,+}(\xi)+ \psi^{L,-}(\xi).
\end{equation}

For the rest of the paper we assume that the L\'evy measure $\nu$ has no atom at $0$.

Below we collect assumptions on $N^\pm(r)$ and the coefficients $a(s,x)$, $b(s,x)$ and $c(s,x,z)$.

\begin{assumption}[Kernel assumptions]\label{N}
    \,\,
    \begin{enumerate}
        \item  Functions $N^\pm $ are right-continuous;

        \item $ N^\pm (r) \asymp \psi^{L,\pm} (1/r)$, $r\in (0,1]$.
    \end{enumerate}

\end{assumption}

Note that the assumption of the right continuity of $N^\pm$ implies that $N^\pm(\tau^\pm(t)) = \frac{1}{t}$; see \cite[Prop.2.3]{EH13}. Clearly, this assumption always holds if $\nu(dz)$ is absolutely continuous with respect to the Lebesgue measure.

It was shown in \cite{IKP25} that under Assumption~\ref{N}.2  the growth bound for $N^{\pm}$ provides the growth bound for $\tau^{\pm}$:
\begin{equation} \label{tau2}
    N^{\pm}(r) \leq r^{-\alpha_\pm}, \; r \in (0,1]
    \; \text{implies}  \;
    \tau^{\pm}(t) \lesssim t^{1/\alpha_\pm}, \; t \in (0,1].
\end{equation}

We also define
\begin{equation}\label{alpha}
    \alpha := \max(\alpha^-, \alpha^+).
\end{equation}

\begin{assumption}[Regularity of coefficients]\label{C}
    Assume that coefficients $ a(s,x) $, $ b(s,x) $, $ c(s,x,z) $ are deterministic measurable functions, satisfying the following conditions.

    \begin{enumerate}
    \item (Lipschitz condition) There exists a constant  $L_{a,b}>0$ and a measurable function $ L_c : [0,T] \times \Real \to \Real_+ $ such that for all $x,y\in \Real$
        \begin{align}
        |a(t,x) - a(t,y)| + |b(t,x) - b(t,y)| &\leq L_{a,b} |x - y|, \quad t \in [0,T], \label{Lipab} \\
        |c(t,x,z) - c(t,y,z)| &\leq L_c(t,z) |x - y|, \quad t \in [0,T],\,  z \in \Real. \label{Lipc}
        \end{align}

    \item (Peano condition)
        There exists a real number $\gamma \in (0,1]$ such that
        \begin{equation}\label{Lipp}
        |a(t,x) - a(s,y)| + |b(t,x) - b(s,y)| \leq L_{a,b} \left( |t - s|^\gamma + |x - y| \right), \quad x, y \in \Real, \ s, t \in [0,T].
        \end{equation}
        \end{enumerate}
\end{assumption}

\begin{assumption}\label{Int}(Joint integrability)
Assume that   for some  $p > 1$
\begin{equation}
\overline{L}_c(t,z) := L_c(t,z) \vee |c(t,0,z)| \in L_p(\nu),
\end{equation}
and
\begin{equation}
M_p(\cdot) \in L_{ 1 + \zeta}([0,T]) \quad \text{for some $\zeta \in (0,1]$},
\end{equation}
where
    \begin{equation}\label{Mp}
        M_p(t) :=
        \begin{cases}
            \displaystyle \int_\Real |\overline{L}_c(t,z)|^p \, \nu(dz), \;  1  < p < 2, \\[1em]
            \displaystyle \int_\Real |\overline{L}_c(t,z)|^p \, \nu(dz) + \left( \int_\Real |\overline{L}_c(t,z)|^2 \, \nu(dz) \right)^{\frac{p}{2}}, \; p \ge 2.
        \end{cases}
    \end{equation}
\end{assumption}
\begin{assumption}[Initial condition]\label{Ini}
    We assume that $X_0\in K$,
    where $K\subset \Real$ is  compact.
\end{assumption}

\begin{assumption}[Boundedness]\label{NC}
There exists $\sigma\in (0,1)$ such that for all $x\in \Real$, $t\in [0,T]$, $|z|\leq 1$ we have
    \begin{equation}
        |c(t,x,z)|\lesssim t^{-\sigma} (1+|x|)|z|.
    \end{equation}
\end{assumption}

\section{Dynamic Cutting Method}

In this subsection we describe the procedure of simulation of a L\'evy process by means of ``dynamic cutting'' of the support of the respective L\'evy measure.  Here we only sketch the argument and refer to \cite{IKP25} for details.

 We delete the   small-jump part  and its compensator by the following dynamic cutting:
\begin{equation}\label{LI11}
    \begin{split}
        Z_{h,\eps}^{1,DC} (t)&  :=     m_{h,\eps}^{DC}(t) +  Z^{CP,+}_{h,\eps}(t) + Z^{CP,-}_{h,\eps}(t),
        \quad t>0,
    \end{split}
\end{equation}
where
\begin{equation}\label{CPpm}
    Z^{CP,+}_{h,\eps}(t):= \sum_{s\leq t } \Delta Z_s \I_{ \Delta Z_s \geq \tau^+((sh)^{\eps } )}, \quad
    Z^{CP,-}_{h,\eps}(t):= \sum_{s\leq t } \Delta Z_s \I_{ \Delta Z_s\leq -\tau^-((sh)^{\eps } )},
\end{equation}
are the upward and downward jumps,  and
\begin{equation}\label{m}
    m_{h,\eps}^{DC}(t) := \int_0^t \left( \int_{-\infty}^{-\tau^-((sh)^\eps)} u \, \nu(du) + \int_{\tau^+((sh)^\eps)}^{\infty} u \, \nu(du)\right)ds
\end{equation}
is the compensated drift.

We describe   the construction of $Z^{CP,+}_{h,\eps}(t)$ only; in order to construct  $Z^{CP,-}_{h,\eps}(t)$, we need to repeat the steps in the construction of  $Z^{CP,+}_{h,\eps}(t)$,  replacing $+$ with $-$ in the respective functions, and multiplying the jumps by $-1$.

Note that the  logarithm of the Laplace transform of    $Z^{CP,+}_{h,\eps}(t)$   (cf. \cite{IKP25} for the proof) equals
\begin{equation}\label{Lap}
    \begin{split}
        -\ln \Ee e^{- r Z^{CP,+}_{h,\eps}(t)}& =
        \int_0^t \!\!\int_{u> \tau^+((sh)^{\eps})} (1- e^{-r u }) \, \nu^+(du)ds\\
        & =  \int_0^\infty  (1- e^{-r u }) \, \mu_{h,\eps}^+(t,du),
    \end{split}
\end{equation}
where
\begin{equation}\label{muh0}
    \mu_{h,\eps}^+(t,du) := \left(\frac{1}{h} (N^+(u))^{-1/\eps}  \wedge t \right)\nu^+(du).
\end{equation}
It follows from representation \eqref{muh0} that
\begin{equation}\label{muh1}
    \mu_{h,\eps}^+(t,du) = \frac{1}{h} \mu_{1,\eps}^+\left(th,du\right)=  t \mu_{th,\eps}^+\left(1,du\right).
\end{equation}
Fix $t>0$. Then  $Z^{CP,+}_{h,\eps}(t)$ is the Compound Poisson-type r.v.:
\begin{equation}\label{Zplus}
    Z^{CP,+}_{h,\eps}(t) = \sum_{k=1}^{N^+_t} Z_{T_k}^+,
\end{equation}
where $N^+_t$ is a Poisson r.v. with   intensity
\begin{equation}
    \lambda_{h,\eps}^+(t) = \int_0^t \!\!\int_{u\geq \tau^+((sh)^\eps)} \nu(du)ds = \mu_{h,\eps}^+(t,\Real_+),
\end{equation}
and $Z_{T_k}^+$, $k\geq 1$, are i.i.d. r.v.'s such that $Z_t^+$ has the distribution function
\begin{equation}\label{Fplus}
    F_{th,\eps}^+(x) := \frac{1}{\lambda_{h,\eps}^+(t)  }\int_{0}^x \mu_{h,\eps}^+(t,du),
\end{equation}
 and $T_k$, $k \ge 1$ denote the time points of the underlying Poisson process $N^+_t$.

Note that by \eqref{muh0} this distribution function depends on $th$, and the function   $\lambda_{h,\eps}^+(t)$  can be calculated explicitly:
\begin{equation}\label{lam1}
    \lambda_{h,\eps}^+(t)
    =\int_0^t \nu^+\{u: \, u\geq \tau^+((sh)^{\eps })\}ds =
    \int_0^t \left(\frac{1}{h s}\right)^{\eps}ds = \frac{t^{1-\eps}}{1-\eps} \left(\frac{1}{h }\right)^{\eps},
\end{equation}
as well as its inverse:
\begin{equation}\label{inv-lam}
    \bigl(\lambda^{\pm}\bigr)^{-1}(t) = \Bigl(t\,(1-\eps)\,h^{\eps}\Bigr)^{\frac{1}{1-\eps}}.
\end{equation}
One can get a more user-friendly expression for the distribution function of $Z_k^+(t)$:
\begin{equation}
    F_{th,\eps}^{\pm}(x)
    = \begin{cases}
        \displaystyle \left(\frac{1}{t\,h}\right)^{1-\eps} \eps \,
        \Bigl(N^\pm(x)\Bigr)^{\tfrac{\eps-1}{\eps}},
        & N^{\pm}(x) \ge (t\,h)^{-\eps}, \\[1.2em]
        \displaystyle 1 - (1-\eps)\,(t\,h)^{\eps}\,N^{\pm}(x),
        & N^{\pm}(x) < (t\,h)^{-\eps}.
    \end{cases}
\end{equation}
The corresponding inverse function is
\begin{equation}
    \bigl(F_{th,\eps}^\pm\bigr)^{-1}(u)
    = \begin{cases}
        \displaystyle
        \tau^\pm\!\Bigl((t\,h)^{\eps}\,\bigl(\tfrac{u}{\eps}\bigr)^{-\tfrac{\eps}{\eps - 1}}\Bigr),
        & u \le \eps, \\[1em]
        \displaystyle
        \tau^\pm\!\Bigl(\tfrac{(1 - \eps)\,(t\,h)^{\eps}}{1 - u}\Bigr),
        & u > \eps.
    \end{cases}
\end{equation}
In order to simulate the jump times of  the  inhomogeneous Poisson process $N^+_t$ we observe  that the $i$-th jump time $T_i^+$ has the same distribution as $ (\lambda^+)^{-1}(\Gamma_i)$, where $\Gamma_i$ is the  jump time of the $i$-th jump of the Poisson process with intensity 1, see \cite{Ci75}.

We  provided  Algorithms \eqref{j-time} and \eqref{j-size}  for generating  jump times and sizes in Section~\ref{S-Alg}.

\section{Euler–Maruyama Approximation and Main Results}

Let us describe the Euler scheme for generating the solution to \eqref{BM_SDE}. We consider two models. In the first model we write the Euler scheme for  \eqref{BM_SDE}, in which the  driven L\'evy process is substituted by  that  introduced in \eqref{LI11}.  In this case \eqref{BM_SDE}  transforms into
\begin{equation}\label{SDE22}
    X_t = X_0 + \int_0^t a(s, X_s) \, ds + \int_0^t b(s, X_s) \, dB_s + \int_0^t \!\!\int_{\Real \backslash B(s, h, \eps) } c(s, X_{s-}, z) \tN (ds, dz),
\end{equation}
where
\begin{equation}\label{Bsh}
    B(s,h,\eps) := \big\{ z: -\tau^-((sh)^\eps) < z < \tau^+((sh)^\eps) \big\}.
\end{equation}
In the second model we substitute the integral over the small-jump part in  \eqref{BM_SDE} with the integral w.r.t. the Brownian motion (independent from $B_s$), i.e.
\begin{equation}\label{SDE23}
    \begin{split}
    X_t &= X_0 + \int_0^t a(s, X_s) \, ds + \int_0^t b(s, X_s) \, dB_s + \int_0^t \!\!\int_{\Real \backslash B(s, h, \eps) } c(s, X_{s-}, z) \tN (ds, dz) \\
    &\qquad +
    \int_0^t \!\!\int_{B(s, h, \eps)} c^2(s,X_{s-},z) \, \nu(dz) \, dW_s,
    \end{split}
\end{equation}
where $W_t$ is a Brownian motion, independent from $B_t$.
Below we provide the Euler-Maruyama scheme for the transformed equations \eqref{SDE22} and \eqref{SDE23}, respectively.

Let $\pi_n:= \{0=t_0<t_1<\dots<t_{n-1}<t_n=T\}$ be a partition of the interval $[0,T]$. Such time points are called regular \emph{grid times}. In addition, the process has large jumps which occur at \emph{jump times}. Such points are illustrated as blue and red dots respectively in Figure~\ref{time_grid}   below.
\begin{figure}[h]
    \centering
    \includegraphics[width=0.9\textwidth]{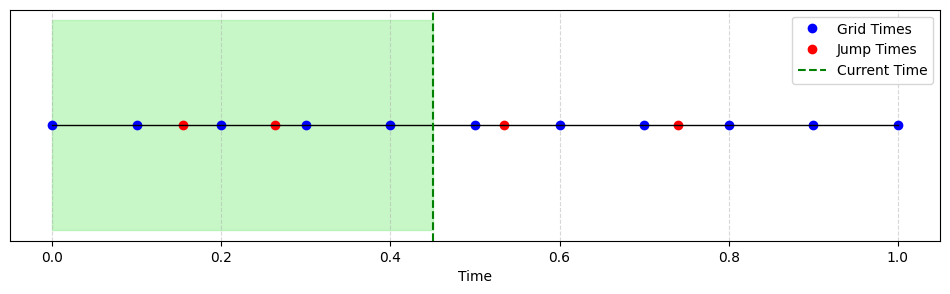}
    \caption{Grid and Jump Times}
    \label{time_grid}
\end{figure}

Therefore, we no longer have a classic Euler-Maruyama situation with equal steps $\Delta := \frac{T}{n}$, because jump times occur at random according to the inhomogeneous Poisson process rate \eqref{lam1} and are sampled using $T_i^\pm \sim (\lambda^\pm)^{-1}(\Gamma_i)$, where $\Gamma_i$ is a sum of i.i.d. exponential random variables with scale 1. Denote a set of jump times as $S^\pm := \{S^+, S^-\}$.

The total grid is then constructed as merged and sorted set of both grid and jump times, where jump points include both positive and negative jumps. Since the total grid consists of several components, we define $\rho(t)$ as the number of merged grid and jump times occurred before time $t$.

Denote the variance of small jumps between two consecutive time points $t_{i-1}$ and $t_i$ as
\begin{equation}\label{sig}
    \sigma_i^2(x) := \int_{t_{i-1}}^{t_i} \int_{B(s, h, \eps)} c^{2}(s,x,z) \, \nu(dz) \, ds.
\end{equation}

Let $\Delta_i := t_i - t_{i-1}, \; i = 1, \dots, \rho(T)$. Then we can write two types of Euler-Maruyama schemes: the one with omitted small jumps and the second one in which small jumps are replaced with the Gaussian approximation:
\begin{align}
    X_{t_i}^{1,\tau} &:= X_{t_{i-1}}^{1,\tau} + a(t_{i-1},X_{t_{i-1}}^{1,\tau}) \Delta_i + b(t_{i-1}, X_{t_{i-1}}^{1,\tau}) \sqrt{\Delta_i} \xi_i + C_{t_{i-1}}(X_{t_{i-1}}^{1,\tau}), \label{Eu1} \\
    X_{t_i}^{2,\tau} &:= X_{t_{i-1}}^{2,\tau} +
    \underbrace{a(t_{i-1},X_{t_{i-1}}^{2,\tau}) \Delta_i}_{\textbf{drift}} +
    \underbrace{b(t_{i-1}, X_{t_{i-1}}^{2,\tau}) \sqrt{\Delta_i} \xi_i}_{\textbf{diffusion}} +
    \underbrace{C_{t_{i-1}}(X_{t_{i-1}}^{2,\tau})}_{\textbf{large jumps}} +
    \underbrace{\sigma_i(X_{t_{i-1}}^{2,\tau})\zeta_i}_{\textbf{small jumps}}, \label{Eu2}
\end{align}
where
\begin{itemize}
    \item $C_{t_i}(x) := c(t_i,x,Z_{t_i}^\pm) \mathbb{I}_{\{t_i \in S\}}$ -- large-jump increment if a jump occurred at time $t_i$,
    \item $(\xi_i)_{1 \leq i \leq \rho(T)}$, $(\zeta_i)_{1\leq i\leq \rho(T)}$ -- sequences of i.i.d. standard Gaussian random variables,
    \item $Z_{t_i}^\pm \sim F_{t_i^\pm h,\eps}^\pm$ is the large jump of the process at time $t_i$.
\end{itemize}

We put $X_t^{k,\tau}= X_{t_i}^{ k,\tau}$, $k=1,2$, for $t\in (t_{i-1},t_i)$.

Below, we formulate our main result.  Recall the definition of
$\alpha$, cf. \eqref{alpha}.
\begin{theorem}\label{tstr}
    Suppose that Assumptions \ref{N} -- \ref{NC} are satisfied.

    \begin{enumerate}
        \item Let $p > 1$. Then for any $\eps^*\in (0,\frac{1}{2p})$
   \begin{equation}\label{str-eq2}
    \Bigl\|\sup_{t\in [0,T]} |X_t - X_t^{2,\tau}|\Bigr\|_{L^p(\Omega)}
    \lesssim n^{-\left\{ \gamma \wedge \frac{2\zeta}{p(1+\zeta)} \right\}} +  n^{ -\frac{2-\alpha}{2\alpha} - \frac{1}{2p} + \eps^*}.
   \end{equation}

    \item Let  $p >  \max(1,\alpha)$  and denote $p^* = p \wedge 2$. For $p\geq 2$ we assume in addition   that $\sigma, \eps$ and $p$ satisfy the inequality  $\sigma+\eps p/2 < 1$.  Then
    \begin{equation}\label{str-eq1}
        \Bigl\|\sup_{t\in [0,T]} |X_t - X_t^{1,\tau}|\Bigr\|_{L^p(\Omega)}
        \lesssim n^{-\left\{ \gamma \wedge \frac{2\zeta}{p(1+\zeta)} \right\}} + n^{-\frac{p^*- \alpha}{p^*\alpha }}.
    \end{equation}
    \end{enumerate}
\end{theorem}

\begin{remark}
In the dynamic cutting (DC) approach, the intensity of positive (or negative) large jumps is defined by $\lambda_{DC}^+(t)$ (see~\eqref{lam1}). Choosing an optimal parameter $h:=n^{-1/\eps}$ and taking $T:=1$  yields that
$\lambda_{DC}^+(1) = \frac{n}{1-\eps} $, meaning that the expected number of large jumps in the DC scheme is of order $O(n)$, independent of the tail behavior of the L\'evy measure.

In contrast, consider the classic Asmussen-Rosi\'nski (AR) approach where the threshold is fixed. In many models one assumes that the L\'evy measure is absolutely continuous with respect to the Lebesgue measure, and its density behaves for small $z$ as
\[
    \nu(z) \sim \frac{C}{|z|^{1+\alpha}}, \quad \alpha\in (0,2).
\]
Thus, we can calculate the intensity similarly:
\[
    \lambda_{AR}^+(1) =  \int_0^1 \!\! \int_{\eps}^\infty \nu(dz) \, ds \sim  \frac{C}{\alpha \eps^\alpha}.
\]
One may set $\eps = \frac{1}{n}$, which is optimal according to \cite{BM24}. Then we obtain
$\lambda_{AR}^+(1) \sim n^\alpha$.

Consequently, in the AR approach the expected number of large jumps is of order $n^\alpha$. Therefore, the total contribution to the small-jump error (which consists of both the grid part and the contribution from large jumps) is of order $n + n^\alpha$.
In particular, if $\alpha\in (0,1)$ then $n + n^\alpha  \approx  n$, while if $\alpha\in [1,2)$ the contribution is dominated by the $n^\alpha$ term.

To summarize,  the AR method  overestimates  the  number of   small jumps for big $\alpha$ and underestimates it for the small ones, whereas  the DC approach yields an expected number of large jumps of order $n$ irrespective of $\alpha$. This makes the DC scheme more robust in controlling the strong error when dealing with L\'evy processes with higher jump activity.
As we will see in Section~\ref{Sim}, AR method gives bigger error than the DC method, and this error increases with the increase of $\alpha$.
\end{remark}

\section{Proofs of Convergence Results}

We start by proving \eqref{str-eq2}. First, we decompose the solution to \eqref{BM_SDE} as
\begin{align*}
    X_t = X_0
    &+ \int_0^t a(s, X_s) \, ds
     + \int_0^t b(s, X_s) \, dB_s \nonumber \\
    &+ \int_0^t \!\! \int_{B(s, h, \eps)} c(s, X_{s-}, z) \, \widetilde{N}(ds, dz)
     + \int_0^t \!\! \int_{\Real \setminus B(s, h, \eps)} c(s, X_{s-}, z) \, \tN(ds, dz).
\end{align*}
Let $ \eta(t) := \sup{\{t_n: t_n \leq t} \}$, and recall that $\rho(t)$ denotes the number of merged grid and jump
times occurred before time $t$. Note that $\eta(t)= t_{i-1}$ for each $t\in [t_{i-1}, t_{i}), \; i \geq 1$.
\begin{equation}\label{X_tilde}
\widetilde{X}_t = \int_0^t a(\eta(s), \widetilde{X}_{\eta(s)}) \, ds + \int_0^t b(\eta(s), \widetilde{X}_{\eta(s)}) \, dB_s + \int_0^t \!\!\int_{-\infty}^{+\infty} c(s, \widetilde{X}_{\eta(s)}, z) \, \tN(ds, dz)
\end{equation}
is the Euler-Peano scheme for \eqref{BM_SDE} with frozen coefficients.

\subsection{Scheme 2: Gaussian Approximation of Small Jumps}
We can rewrite \eqref{Eu2} using the step function $\eta(s)$:
\begin{equation}\label{Eu22}
    \begin{split}
X_t^{2,\tau}&= \int_0^t a(\eta(s), X_{\eta(s)}^{2,\tau}) \, ds + \int_0^t b(\eta(s), X_{\eta(s)}^{2,\tau}) \, dB_s +\sum_{i=1}^{\rho(t)}\sigma_i(X_{t_{i-1}}^{2,\tau})\zeta_i \\ & \qquad \qquad + \int_0^t \int_{\Real \setminus B(s, h, \eps)} c(s, X_{\eta(s)}^{2,\tau}, z)\tN(ds, dz).
\end{split}
\end{equation}

Split
\begin{align*}
\left\| \sup_{0 \leq i \leq n} \left| X_{t_i} - X_{t_i}^{2,\tau} \right| \right\|_{L^p(\Omega)} = &\left\| \sup_{0 \leq i \leq n} \left| X_{t_i} - \widetilde{X}_{t_i} + \widetilde{X}_{t_i} - X_{t_i}^{2,\tau} \right| \right\|_{L^p(\Omega)} \\
\leq &\left\| \sup_{0 \leq i \leq n} \left| X_{t_i} - \widetilde{X}_{t_i} \right| + \left| \widetilde{X}_{t_i} - X_{t_i}^{2,\tau} \right| \right\|_{L^p(\Omega)} \\
\leq &\left\| \sup_{t\in [0,T]} \left| X_t - \widetilde{X}_t \right| \right\|_{L^p(\Omega)} + \left\| \sup_{0 \leq i \leq n} \left| \widetilde{X}_{t_i} - X_{t_i}^{2,\tau} \right| \right\|_{L^p(\Omega)}.
\end{align*}

Under assumptions \ref{C}, \ref{NC} and \ref{Int} of Theorem~\ref{tstr} we have the following estimate, which is taken from \cite[Lem.3.1]{BM24}:
\begin{equation}\label{lem31_est}
 \left\| \sup_{t \in [0,T]} |X_t - \widetilde{X}_t| \right\|_{L^p(\Omega)} \lesssim n^{- \gamma \wedge \frac{2\zeta}{p(1+\zeta)} }.
\end{equation}
The proof of the above lemma extends to $p  >  1$ by switching Kunita's BDG to Novikov's theorem \ref{Novikov} which is valid for  $p \in [1,2]$. To prove \eqref{str-eq2} we show that
\begin{equation}
\sup_{\pi_n } \left\| \sup_{0 \leq i \leq  \rho(T)} \left| \widetilde{X}_{t_i} - X_{t_i}^{2,\tau} \right| \right\|_{L^p(\Omega)}\lesssim \tau(n^{-1}) n^{\frac{1}{2} - \frac{\delta}{(3+2\delta)\,p}}.
\end{equation}

By the definition of $\eta(t)$, the difference between $ \widetilde{X}_{\eta(t)} $ and $ \widehat{X}^{2,\tau}_{\eta(t)} $, $t\in [0,T]$, is equal to
\begin{align*}
\widetilde{X}_{\eta(t)} - X_{\eta(t)}^{2,\tau} &= \int_0^{\eta(t)} \left( a(\eta(s), \widetilde{X}_{\eta(s)}) - a(\eta(s), X_{\eta(s)}^{2,\tau}) \right) ds \\
&\quad + \int_0^{\eta(t)} \left( b(\eta(s), \widetilde{X}_{\eta(s)}) - b(\eta(s), X_{\eta(s)}^{2,\tau}) \right) dB_s \\
&\quad + \int_0^{\eta(t)} \int_{\Real \setminus B(s, h, \eps)} \left( c(s, \widetilde{X}_{\eta(s)}, z) -c(s, X_{\eta(s)}^{2,\tau}, z) \right) \tN(ds, dz) \\
&\quad + \sum_{i=1}^{\rho(t)} \left( \int_{t_{i-1}}^{t_i} \int_{B(s, h, \eps)} c(s, \widetilde{X}_{\eta(s)}, z) \tN(ds, dz)- \sigma_i(X_{t_{i-1}}^{2,\tau})\zeta_i \right).
\end{align*}
Applying the Minkowski's inequality we can decompose the $ L^p $-norm for $p \ge 1$  of the sum into four term,
\begin{equation}
\left\|X_{\eta(t)}^{2,\tau}- \widetilde{X}_{\eta(t)} \right\|_{L^p(\Omega)} \leq \underbrace{\mathrm{(drift)}}_{\textbf{(A)}}
    + \underbrace{\mathrm{(diffusion)}}_{\textbf{(B)}}
    + \underbrace{\mathrm{(large~jumps)}}_{\textbf{(C)}}
    + \underbrace{\mathrm{(small~jumps)}}_{\textbf{(D)}},
\end{equation}
which we estimate separately.

\noindent \textbf{Estimation of (A).}
Applying subsequently the Minkowski's Integral Inequality and the Lipschitz assumption \eqref{Lipab} we derive
\begin{align*}
\textbf{(A)} &:= \left\|\int_0^{\eta(t)} \left( a(\eta(s), \widetilde{X}_{\eta(s)}) - a(\eta(s), X_{\eta(s)}^{2,\tau}) \right) ds \right\|_{L^p(\Omega)} \\
&\leq \int_0^{\eta(t)} \left\| a(\eta(s), \widetilde{X}_{\eta(s)}) - a(\eta(s), X_{\eta(s)}^{2,\tau}) \right\|_{L^p(\Omega)} ds\\
 &\leq L_{a,b} \int_0^{\eta(t)}  \left\| \widetilde{X}_{\eta(s)}- X_{\eta(s)}^{2,\tau} \right\|_{L^p(\Omega)} \, ds \\
&\lesssim \int_0^{\eta(t)} \sup_{0 \leq s \leq \eta(t)} \left\| \widetilde{X}_{\eta(s)}- X_{\eta(s)}^{2,\tau} \right\|_{L^p(\Omega)} ds.
\end{align*}

\noindent \textbf{Estimation of (B).}
For the second term, we use the Burkholder-Davis-Gundy (BDG) inequality  for continuous martingales  (see  \ref{bdg}) and again the Lipschitz assumption \eqref{Lipab}:
\begin{align*}
\textbf{(B)} &:= \left\| \int_0^{\eta(t)} \left( b(\eta(s), \widetilde{X}_{\eta(s)}) - b(\eta(s), X_{\eta(s)}^{2,\tau}) \right)dB_s \right\|_{L^p(\Omega)} \\
&\leq \left( C_p^{BDG} \right)^{\frac{1}{p}}
\left( \int_0^{\eta(t)} \left\| b(\eta(s), \widetilde{X}_{\eta(s)}) - b(\eta(s), X_{\eta(s)}^{2,\tau}) \right\|_{L^p(\Omega)}^2 ds \right)^{\frac{1}{2}}\\
&\leq \left( C_p^{BDG} \right)^{\frac{1}{p}} L_{a,b} \left( \int_0^{\eta(t)}  \left\| \widetilde{X}_{\eta(s)}- X_{\eta(s)}^{2,\tau} \right\|_{L^p(\Omega)}^2 ds \right)^{\frac{1}{2}} \\
&\lesssim \left( \int_0^{\eta(t)} \sup_{0 \leq s \leq \eta(t)} \left\|\widetilde{X}_{\eta(s)}- X_{\eta(s)}^{2,\tau} \right\|_{L^p(\Omega)}^2 \, ds \right)^{\frac{1}{2}}.
\end{align*}

\noindent \textbf{Estimation of (C).}
To shorten the notation, denote
\begin{align*}
F(s, z) &:= c(s, \widetilde{X}_{\eta(s)}, z) - c(s, X_{\eta(s)}^{2,\tau} , z).
\end{align*}

We handle the term with the large jumps using Theorem~\ref{Novikov}.

\noindent Case 1. $p \ge 2$. Applying  \eqref{Novikov2}, we get
\begin{align*}
\textbf{(C)} &:= \left\| \int_0^{\eta(t)} \!\!\int_{\Real \setminus B(s, h, \eps)} c(s, \widetilde{X}_{\eta(s)}, z) - c(s, X_{\eta(s)}^{2,\tau} , z) \tN(ds, dz) \right\|_{L^p(\Omega)} \\
&= \left( \E{ \left| \int_0^{\eta(t)} \!\!\int_{\Real \setminus B(s, h, \eps)} F(s, z) \tN(ds, dz) \right|^p } \right)^{\frac{1}{p}} \\
&\leq \left( \E{ \sup_{0 \leq t \leq \eta(t)} \left| \int_0^{\eta(t)} \!\!\int_{\Real \setminus B(s, h, \eps)} F(s, z) \tN(ds, dz) \right|^p } \right)^{\frac{1}{p}} \\
&\lesssim \left( \int_0^{\eta(t)} \E{\left( \int_{-\infty}^{+\infty} \left| F(s, z) \right|^2 \nu(dz)\right)^{p/2}} ds + \E{\int_0^{\eta(t)} \int_{-\infty}^{+\infty} \left| F(s, z) \right|^p \nu(dz) \, ds} \right)^{\frac{1}{p}}.
\end{align*}

\noindent Case 2. $1 < p < 2$.

We apply \eqref{Novikov1}  with $a = p$:
\begin{align*}
    \textbf{(C)}
    &= \left\| \int_0^{\eta(t)} \!\!\int_{\Real\setminus B(s,h,\eps)} F(s,z) \,\tN(ds, dz) \right\|_{L^p(\Omega)} \\
    &\lesssim \E{\left( \int_0^{\eta(t)} \!\!\int_{\Real\setminus B(s,h,\eps)}
    \bigl| F(s,z) \bigr|^{a} \,\nu(dz)\, ds \right)^{p/ a}}^{1/p} \\
    &= \E{\int_0^{\eta(t)} \!\!\int_{\Real\setminus B(s,h,\eps)}
    \bigl| F(s,z) \bigr|^p \,\nu(dz)\, ds}^{1/p} \\
    &\lesssim \E{\int_0^{\eta(t)} \!\!\int_{-\infty}^\infty
    \bigl| F(s,z) \bigr|^p \,\nu(dz)\, ds}^{1/p}.
\end{align*}

Using the Lipschitz assumption \eqref{Lipc} and combining both cases, we obtain
\begin{align*}
\textbf{(C)} &\lesssim \left( \int_0^{\eta(t)} \sup_{0 \leq s \leq \eta(t)} \left\|\widetilde{X}_{\eta(s)}- X_{\eta(s)}^{2,\tau} \right\|_{L^p(\Omega)}^p M_p(s) \, ds \right)^{\frac{1}{p}},
\end{align*}
where the function $M_p(s)$ is defined in \eqref{Mp}.

\noindent \textbf{Estimation of (D).}
Our goal is to find an upper bound for:
\begin{equation}
\textbf{(D)}:= \left\| \sum_{i=1}^{\rho(t)} \left( \int_{t_{i-1}}^{t_i} \int_{B(s, h, \eps)} c(s, \widetilde{X}_{\eta(s)}, z) \tN(ds, dz)- \sigma_i(X_{t_{i-1}}^{2,\tau})\zeta_i \right) \right\|_{L^p(\Omega)}.
\end{equation}
Note that for $s\in [t_{i-1},t_i)$ we have $\eta(s)= t_{i-1}$. Put
\begin{equation}
J_i^{\text{small}} (x) := \int_{t_{i-1}}^{t_i} \int_{B(s, h, \eps)} c(s, x, z) \tN(ds, dz).
\end{equation}
Then we can rewrite $\textbf{(D)}$ as
\begin{equation}\label{D1}
    \textbf{(D)}= \left\| \sum_{i=1}^{\rho(t)} \left( J_i^{\text{small}} (X_{t_{i-1}}^{2,\tau}) - \sigma_i(X_{t_{i-1}}^{2,\tau})\zeta_i \right) \right\|_{L^p(\Omega)}.
\end{equation}

We rely on the approach from \cite{BM24}.

Recall the definition of the Wasserstein (or Kantorovich-Rubinstein) distance of order $p$. Let $\Pp_p$ be the space of probability measures on $\Real$, possessing moments of order $p$. The Wasserstein distance between $\mu_1,\mu_2\in \Pp_p$ is defined in the following way:
\begin{equation}
    \Ww_p(\mu_1,\mu_2) := \inf_{\pi \in \Pp_p} \left(\int_\Real \int_\Real |x-y|^p \pi(dx,dy)\right)^{\frac1p};
\end{equation}
here $p\geq 1$ and the infimum runs over all probability measures $\pi$ on $\Real\times \Real$ with marginals $\mu_1$ and $\mu_2$.

 According to \cite[Th.1.1]{Fo10}, see also \cite{GMC96},  if $p>1$, one can construct the measurable mapping $T_i: \Real\times \Real\mapsto \Real$, which optimally transports $J_i^{\text{small}}(x)$ to an absolutely continuous law,  in particular, a normal law, i.e.  $T_i\bigl(x,J_i^{small}(x)\bigr) \sim\mathcal{N} \bigl(0,\sigma^2_i(x)\bigr)$ $\forall x\in\Real$:
\begin{equation}
    \Ww_p\Bigl(J_i^{small}(x), T_i\bigl(x,J_i^{small}(x)\bigr)\Bigr) = \Ee \left\|J_i^{small}(x)- T_i(x,J_i^{small}(x))\right\|^p.
\end{equation}
\begin{remark}
    When $p=1$, the cost function $\|x-y\|$ is linear and lacks the strict convexity required to guarantee the uniqueness of the optimal transport map. Consequently, our subsequent analysis is restricted to $p>1$, where this uniqueness property holds.
\end{remark}
Let
\begin{align*}
    S_n &:= \sum_{i=1}^n J_i^{\text{small}}(X_{t_{i-1}}^{2,\tau}),
    &&V_n := \sum_{i=1}^n \sigma^2_i(X_{t_{i-1}}^{2,\tau}), \\
    \quad \widetilde{S}_n &:= \frac{S_n}{\sqrt{V_n}},
    &&\zeta_n := \sum_{i=1}^n \sigma_i(X_{t_{i-1}}^{2,\tau}) \eta_i,
\end{align*}
where $\eta_i$, $1\leq i\leq n$, are i.i.d., $\eta_1\sim \mathcal{N}(0,1) $. It is clear that $\zeta_k\sim \mathcal{N}(0, V_k) $, $1\leq k\leq n$.

By \cite[Th.6.4.5]{AD09}, the strong solution to a jump-driven SDE, in particular, the one defined by $J_i^{small}(x)$,  is a Markov process (the proof we refer to is written for $p=2$, but can be easily adapted for $p\geq 2$).

 Let $\Ff_n:= \sigma(S_n)$. By the Markov property, $(S_n,\Ff_n)$ is the sequence of martingale differences, i.e. $\Ff_n\subset \Ff_{n+1}$, $n\geq 0$, and $\Ee[ S_n |\Ff_{n-1}]=0$.
 Moreover, $\Ee |S_n|^2 <\infty$ and $\Cov(J_i^{small}(X_{t_{i-1}}^{2,\tau}), J_{i+k}^{small}(X_{t_{i+k-1}}^{2,\tau})) =0$, $k\geq 1$.
 Thus, variance of $S_n$ can be expressed as the sum of  variances of the martingale differences:
 \[
 \Var \left[S_n\right] = \sum_{i=1}^n \sigma^2_i (X_{t_{i-1}}^{2,\tau}).
 \]
 Note that $(S_k - \zeta_k)_{k\geq 1}$ is an $\Ff_k$--martingale.

Then by the Doob inequality,
 for $p >  1$ we have
 $$
 \left\|\sup_{t\in [0,T]}
  \sum_{i=1}^{\rho(t)} \left( J_i^{\text{small}} (X_{t_{i-1}}^{2,\tau}) - \sigma_i(X_{t_{i-1}}^{2,\tau})\eta_i \right) \right\|_{L^p(\Omega)} \lesssim \left\|
  \sum_{i=1}^{\rho(T)} \left( J_i^{\text{small}} (X_{t_{i-1}}^{2,\tau}) - \sigma_i(X_{t_{i-1}}^{2,\tau})\eta_i \right) \right\|_{L^p(\Omega)}.
  $$

 Using \eqref{D1}, we can apply the normalization and rewrite $\textbf{(D)}$ as
 \begin{equation}\label{D2}
    \textbf{(D)} \lesssim \sqrt{V_{\rho(T)}} \cdot \Ww_p\left( \widetilde{S}_{\rho(T)}, \vartheta \right), \quad \vartheta\in \mathcal{N}(0,1).
 \end{equation}

 The key point of applying the Wasserstein distance is that we can estimate the right-hand side of \eqref{D2}, using Theorem~\ref{bobkov_wp_bound}:
\begin{equation}\label{D3}
\Ww_p\left( \widetilde{S}_{\rho(T)},  \vartheta  \right) \leq \int_{-\infty}^\infty \left| F_{\widetilde{S}_{\rho(T)}}(x) - \Phi(x) \right|^{\frac{1}{p}} dx,
\end{equation}
where $F_{\widetilde{S}_{\rho(T)}}(x)$ is the distribution function  of $\widetilde{S}_{\rho(T)}$ and  $\Phi(x)$ is the distribution function of standard normal variable.

The expression under the integral on the right-hand side can be estimated by applying the martingale version of Esseen’s theorem, see Theorem~\ref{HJ88}:
\begin{equation}
\left| F_{\widetilde{S}_{\rho(T)}}(x) - \Phi(x) \right| \leq \frac{C_\delta \, L_{\rho(T), 2\delta}^{\frac{1}{3 + 2\delta}}}{1 + |x|^{2 + 2\delta}}, \quad
L_{\rho(T), 2\delta} = \frac{1}{V^{1+\delta}_{\rho(T)}} \cdot \sum_{i=1}^{\rho(T)} \E{\left|J_i^{\text{small}}\right|^{2+2\delta}}.
\end{equation}
Expanding \textbf{(D)} gives us
\begin{align*}
    \textbf{(D)} &\leq C_\delta \sqrt{V_{\rho(T)}} L_{\rho(T), 2\delta}^{\frac{1}{(3+2\delta)p}} \int_{\Real} \frac{1}{\left(1 + |x|^{2+2\delta}\right)^{1/p}} \, dx \\
    &\lesssim V_{\rho(T)}^{\frac{1}{2}\left(1 - \frac{2+2\delta}{(3+2\delta)p}\right)} \left( \sum_{i=1}^{\rho(T)} \E{\left|J_i^{\text{small}}\right|^{2+2\delta}} \right)^{\frac{1}{(3 + 2\delta)p}},
\end{align*}
where we used that under the assumption that  $\delta > \frac{p}{2} - 1$ the integral in the first inequality is convergent.

Applying \ref{Novikov}, since $2+2\delta \ge 2$, we obtain
\begin{align*}
    \sum_{i=1}^{\rho(T)} \E{\left|J_i^{\text{small}}\right|^{2+2\delta}}
    &= \sum_{i=1}^{\rho(T)} \E{\left|\int_{t_{i-1}}^{t_i} \int_{B(s, h, \eps)} c(s, \widetilde{X}_{t_{i-1}}, z) \tN(ds, dz)\right|^{2+2\delta}} \\
    &\lesssim \E{\int_0^T \int_{B(s, h, \eps)} \left| c(s, \widetilde{X}_{\eta(s)}, z) \right|^{2+2\delta} \nu(dz) ds} \\
    &\qquad + \sum_{i=1}^{\rho(T)} \E{\left( \int_{t_{i-1}}^{t_i} \int_{B(s, h, \eps)} \left|c(s, \widetilde{X}_{t_{i-1}}, z)\right|^2 \nu(dz) ds\right)^{1+\delta}} =: I_1 + I_2.
\end{align*}

Based on Lemma \ref{int_c_p_1} and \ref{E_c_p_1}, we can find the upper bounds for $I_1$ and $I_2$:
\begin{align*}
    I_1 &\lesssim \tau^{2+2\delta}(h^\eps) h^{-\eps}, \\
    I_2 &\lesssim \sum_{i=1}^{\rho(T)} \left( \int_{t_{i-1}}^{t_i} \tau^2((sh)^\eps) (sh)^{-\eps} s^{-\sigma} \, ds \right)^{1+\delta}.
\end{align*}

By  \eqref{lam1}, the expected number of large jumps is approximately $ C\,h^{-\eps} $. In particular, if we choose $ h = n^{-1/\eps} $, then the expected number of large jumps is of order $ O(n) $. Recall that  $\rho(T)$ represents the merged grid combining both the regular grid and the jump times.  Since the number of jumps is $O(n)$,  we can first analyze the contribution from the regular grid and then multiply the result by a constant.   With these observations, we can now proceed to bound $ I_2 $ further. We have
\begin{align*}
    I_2
    & \lesssim \left( \tau^2 \left( (Th)^\eps \right) h^{-\eps} \left( \frac{T}{n} \right)^{1-\eps} \right)^{1+\delta} \sum_{i=1}^{n} \left( \int_{i-1}^i s^{-\eps-\sigma} \, ds \right)^{1+\delta} \\
    &\lesssim \frac{ \tau^{2+2\delta}(h^\eps)}{\left(h^\eps \, n^{1-\eps}\right)^{1+\delta}} \sum_{i=1}^{n } \left( \int_{i-1}^i s^{-\eps-\sigma}  \, ds \right)^{1+\delta},
\end{align*}
where in the last line we used that $\tau(\cdot)$ is increasing.

Using the inequality
$
    1 - (1 - x)^a \le 2^{1 - a} a x,
$
valid for $ x \in (0, \tfrac{1}{2}] $, $ a \in (0, 1) $, we obtain
\begin{align*}
    \sum_{i=1}^{n} \left( \int_{i-1}^i s^{-\eps} \, ds \right)^{1+\delta}
    &= \sum_{i=1}^{n } \left( \frac{i^{1-\eps} - (i-1)^{1-\eps}}{1-\eps} \right)^{1+\delta} \\
    &= \sum_{i=1}^{n } \left( \frac{i^{1-\eps}}{1-\eps} \left( 1 - \left( 1 - \frac{1}{i} \right)^{1-\eps} \right)\right)^{1+\delta} \\
    &\leq \frac{1}{1-\eps} + \sum_{i=2}^n \left( \frac{i^{1-\eps}}{1-\eps} \cdot \frac{2^\eps(1-\eps)}{i} \right)^{1+\delta} \\
    &\lesssim \sum_{i=1}^n \frac{1}{i^{\eps(1+\delta)}} \\
    &\lesssim n^{1-\eps(1+\delta)}.
\end{align*}
Summarizing, we get
\begin{align*}
    I_2
    &\lesssim  \frac{ \tau^{2+2\delta}(h^\eps)}{\left(h^\eps \, n^{1-\eps}\right)^{1+\delta}} \, n^{1-(\eps+\sigma)(1+\delta)} = \frac{\tau^{2+2\delta}(h^\eps)}{h^{\eps(1+\delta)}} \, n^{-\delta - \sigma(1+\delta)}.
\end{align*}
Thus,
$$
\sum_{i=1}^{n} \E{\left|J_i^{\text{small}}\right|^{2+2\delta}} \lesssim I_1+I_2 \lesssim  \tau^{2+2\delta} (h^\eps) h^{-\eps}[1+ h^{-\eps \delta }n^{-\delta- \sigma(1+\delta)}].
$$
Choose $h= n^{-1/\eps}$. Then
\begin{equation}
    \sum_{i=1}^{n} \E{\left|J_i^{\text{small}}\right|^{2+2\delta}} \lesssim \tau^{2+2\delta} (n^{-1}) n.
\end{equation}
Using Lemma \ref{E_c_p_1}, we can obtain an upper bound for $V_n$:
\begin{align*}
    V_n^\theta
    &\lesssim \tau^{2\theta}((Th)^{\eps})  \left(T(Th)^{-\eps}\right)^{\theta} \lesssim \tau^{2\theta}(h^\eps) h^{-\eps\theta},
\end{align*}
where
\begin{equation*}
    \theta := \frac{1}{2}\left(1-\frac{2+2\delta}{(3+2\delta)p}\right) > 0.
\end{equation*}

Summarizing, we get
\begin{align*}
  \textbf{(D)} \lesssim \tau(n^{-1}) n^{\frac12 - \frac{\delta}{(3+2\delta)p}}.
\end{align*}

\textbf{Completing the proof.}
Denote
\begin{equation}
\Phi_n(t) := \sup_{0 \leq s \leq \eta(t)} \left\| \widetilde{X}_{\eta(s)} - X^{2,\tau}_{\eta(s)} \right\|_{L^p(\Omega)},
\end{equation}
where the index $n$ emphasizes that  the supremum implicitly depends on $n$  via $\eta(s)$.
Combining the bounds for terms (A)--(D), we have:
\begin{equation}
\Phi_n(t) \lesssim \int_0^{t} \Phi_n(s) ds + \left( \int_0^{t} \Phi^2_n(s) ds \right)^{1/2} + \left( \int_0^{t} \Phi^p_n(s) ds \right)^{1/p} + r(n),
\end{equation}
where $r(n):= \tau(n^{-1}) n^{\frac12 - \frac{\delta}{(3+2\delta)p}}$.

Using Gronwall's inequality (see \ref{gronwall}), we get
\begin{equation}
    \Phi_n(t) \leq C e^{Ct} \cdot r(n)
    \; \implies \;
    \sup_{0 \le t \le T} \Phi_n(t) \lesssim r(n).
\end{equation}

Adding the error from the Euler-Peano approximation, we finally arrive at
\begin{equation}
    \left\| \sup_{0 \leq t \leq T} \left| X_t - X^{2,\tau}_t \right| \right\|_{L^p(\Omega)}
     \lesssim n^{-\left\{ \gamma \wedge \frac{2\zeta}{p(1+\zeta)} \right\}} + \tau(n^{-1}) n^{\frac{1}{2} - \frac{\delta}{(3+2\delta)\,p}}.
\end{equation}

Note that the function
 $f(\delta) = \frac{1}{2} - \frac{\delta}{(3+2\delta)\,p}$  is strictly decreasing  in $\delta$, which can be chosen arbitrarily large. Hence,  choosing $\delta$ large we arrive at
\begin{equation}
    \left\| \sup_{0 \leq t \leq T} \left| X_t - X^{2,\tau}_t \right| \right\|_{L^p(\Omega)}
     \lesssim n^{-\left\{ \gamma \wedge \frac{2\zeta}{p(1+\zeta)} \right\}} + \tau(n^{-1}) n^{\frac{1}{2} - \frac{1}{2p} + \eps^*}
\end{equation}
 for any $0 < \eps^* < \frac{1}{2p}$. Finally,  by \eqref{tau2} we have $\tau(n^{-1})n^{1/2} \lesssim  n^{\frac{\alpha-2}{2\alpha}} \to 0$, $n\to \infty$.

\subsection{Scheme 1: Omitted Small Jumps}
We consider the \emph{no-small-jumps} approximation
\begin{equation*}
    X_t^{1,\tau} = X_0
    + \int_0^t a\bigl(s, X_s^{1,\tau}\bigr)\,ds
    + \int_0^t b\bigl(s, X_s^{1,\tau}\bigr)\,dB_s
    + \int_0^t \!\!\int_{\Real\setminus B(s,h,\eps)} c\bigl(s, X_{s-}^{1,\tau}, z\bigr) \, \tN(ds,dz),
\end{equation*}
where the small jumps of size $ B(s,h,\eps)$) are omitted.

Our goal is to find the upper bound for the strong error
\[
\Bigl\|\sup_{t\in [0,T]} \bigl| X_t - X_t^{1,\tau} \bigr|\Bigr\|_{L^p(\Omega)}, \quad \max(\alpha, 1) <  p < \infty.
\]
As in the $X^{2,\tau}_t$ case, the error is decomposed into two parts, where $\widetilde{X}_t$ is defined in \eqref{X_tilde}:
\begin{equation*}
    \left\| \sup_{t\in [0,T]} \left| X_t - X_t^{1,\tau} \right| \right\|_{L^p(\Omega)} \leq
    \left\| \sup_{t\in [0,T]} \left| X_t - \widetilde{X}_t \right| \right\|_{L^p(\Omega)} +
    \left\| \sup_{t\in [0,T]} \left| \widetilde{X}_{\eta(t)} - X_{\eta(t)}^{1,\tau} \right| \right\|_{L^p(\Omega)}.
\end{equation*}

The upper bound on the first part is already proved in  \eqref{lem31_est}.

Split the difference in the second term into four components:
\[
    \widetilde{X}_{\eta(t)} - X_{\eta(t)}^{1,\tau} = \underbrace{\mathrm{(drift)}}_{\textbf{(A)}}
    + \underbrace{\mathrm{(diffusion)}}_{\textbf{(B)}}
    + \underbrace{\mathrm{(large~jumps)}}_{\textbf{(C)}}
    + \underbrace{\mathrm{(small~jumps)}}_{\textbf{(D)}}.
\]

The drift, diffusion, and large-jump terms \textbf{(A)}, \textbf{(B)}, \textbf{(C)} are treated exactly as in the proof of \eqref{str-eq2}. The only new piece is \textbf{(D)}, the omitted small jumps:
\[
\left\|\int_0^{\eta(t)} \!\!\int_{B(s,h,\eps)} c(s, \widetilde{X}_{\eta(s)}, z) \, \tN(ds,dz)\right\|_{L^p(\Omega)}.
\]

\noindent\emph{Case 1: $p\ge2$.}

By the Kunita BDG (cf.  \eqref{Novikov2}), we
\begin{align*}
\Biggl\| &\int_0^{\eta(t)} \!\!\int_{B(s,h,\eps)} c(s,\widetilde{X}_{\eta(s)},z)\,\tN(ds,dz) \Biggr\|_{L^p(\Omega)}^p \\
&\lesssim \int_0^{\eta(t)} \E{\left( \int_{B(s,h,\eps)} |c(s,\widetilde{X}_{\eta(s)},z)|^2 \nu(dz) \right)^{p/2} } ds \\
&\quad + \E{\int_0^{\eta(t)} \!\!\int_{B(s,h,\eps)} |c(s,\widetilde{X}_{\eta(s)},z)|^p \nu(dz) \,ds}.
\end{align*}

According to Lemma \ref{int_c_p_1} and assumption that  $\sigma+\eps p/2 < 1$:
\begin{align*}
    \int_0^{\eta(t)} \E{\left( \int_{B(s,h,\eps)} |c(s,\widetilde{X}_{\eta(s)},z)|^2 \nu(dz) \right)^{p/2} } ds
    &\lesssim \int_0^{\eta(t)} \frac{(\tau((sh)^\eps))^p s^{-\sigma}}{(sh)^{\eps p/2}} \, ds \\
    &\lesssim \frac{(\tau(h^\eps))^p}{h^{\eps p/2}} \int_0^T s^{-\sigma-\eps p/2} \, ds \\
    &\lesssim \frac{(\tau(h^\eps))^p}{h^{\eps p/2}}.
\end{align*}

According to Lemma \ref{E_c_p_1}, the second term is bounded by
\[
    \E{\int_0^{\eta(t)} \!\!\int_{B(s,h,\eps)} |c(s,\widetilde{X}_{\eta(s)},z)|^p \nu(dz) \,ds}
    \lesssim \frac{\big(\tau(h^\eps)\big)^p}{h^\eps}.
\]

Since $0 < h < 1$, then $\frac{1}{h^{\eps p/2}} > \frac{1}{h^\eps}$. Therefore,
\[
    \frac{(\tau(h^\eps))^p}{h^{\eps p/2}} + \frac{(\tau(h^\eps))^p}{h^\eps} \lesssim \frac{(\tau(h^\eps))^p}{h^{\eps p/2}}.
\]

Taking the $1/p$ power, choosing $h = n^{-1/\eps}$ and applying the extended Gronwall lemma \ref{gronwall}, we obtain
\[
    \Bigl\|\sup_{t\in[0,T]}|\widetilde{X}_{\eta(t)}-X_{\eta(t)}^{1,\tau}|\Bigr\|_{L^p(\Omega)}
    \lesssim \tau(n^{-1}) \sqrt{n}.
\]

\noindent\emph{Case 2: $\max(\alpha, 1) < p<2$.}

We apply Novikov's inequality \eqref{Novikov1} with $a=p$ and Lemma~\ref{E_c_p_2} to obtain
\begin{align*}
    \E{\sup_{t\in[0,T]}\left|\int_0^t \!\!\int_{B(s,h,\eps)} c(s,\widetilde{X}_{\eta(s)},z)\,\tN(ds,dz)\right|^p}
    &\lesssim \E{\int_0^T \!\!\int_{B(s,h,\eps)} \bigl|c(s,\widetilde{X}_{\eta(s)},z)\bigr|^p\,\nu(dz)\,ds} \\
    &\lesssim  \big( \tau(h^\eps) \big)^{p-\alpha}.
\end{align*}

Taking the $1/p$ power and applying the extended Gronwall lemma \ref{gronwall} for $h = n^{1/\eps}$ yields
\[
    \Bigl\|\sup_{t\in[0,T]}\bigl|\widetilde{X}_{\eta(t)}-X_{\eta(t)}^{1,\tau}\bigr|\Bigr\|_{L^p(\Omega)}
    \lesssim (\tau(n^{-1}))^{\frac{p-\alpha}{p}}.
\]

\section{Simulation Algorithms}\label{S-Alg}

The implementation of the algorithms and simulation of the examples described below can be found at: \url{https://github.com/d-platonov/levy-type}.

\begin{algorithm}[H]
    \caption{Simulation of jump times up to time $ T $}
    \label{j-time}
    \begin{algorithmic}[1]
        \State Initialize $\mathcal{T}^{\pm} \gets \emptyset$ \Comment{Set of jump times}
        \State $S \gets 0$ \Comment{Cumulative sum of exponential random variables}
        \State Generate  $E \sim \mathrm{Exp}(1)$
        \State Update $S \gets S + E$
        \State $t_{\mathrm{new}} \gets \bigl(\lambda^{\pm}\bigr)^{-1}(S)$
        \While{$t_{\mathrm{new}} < T$}
        \State $\mathcal{T}^{\pm} \gets \mathcal{T}^{\pm} \cup \{t_{\mathrm{new}}\}$
        \State Generate $E' \sim \mathrm{Exp}(1)$; update $S \gets S + E'$
        \State $t_{\mathrm{new}} \gets \bigl(\lambda^{\pm}\bigr)^{-1}(S)$
        \EndWhile
        \State \Return $\mathcal{T}^{\pm}$
    \end{algorithmic}
\end{algorithm}

\begin{algorithm}[H]
    \caption{Sampling of large jump sizes for given jump times}
    \label{j-size}
    \begin{algorithmic}[1]
        \Require List of jump times $\mathcal{T}^{\pm}$
        \State Initialize $\mathcal{Z}^{\pm} \gets \emptyset$
        \For{$t \in \mathcal{T}^{\pm}$}
        \State Generate $U \sim \mathrm{Uniform}(0,1)$
        \State $Z \gets \pm \bigl(F^{\pm}\bigr)^{-1}(U)$
        \State $\mathcal{Z}^{\pm} \gets \mathcal{Z}^{\pm} \cup \{Z\}$
        \EndFor
        \State \Return $\mathcal{Z}^{\pm}$
    \end{algorithmic}
\end{algorithm}

\begin{algorithm}[H]
\caption{Simulate a single L\'evy-type path}
\label{bench_path}
\begin{algorithmic}[1]
\Require Simulation parameters: $n$, $T$, $\eps$, $h$, $x_0$, $\alpha$ (optional)
\Require SDE coefficients: drift $a(t,x)$, diffusion $b(t,x)$, jump $c(t,x,z)$
\Require Inverse functions: $\lambda^{-1}(u)$ (jump intensity) and $F^{-1}(u,t)$ (large-jump CDF)
\Require Expression for small-jump variance: $\sigma^2_{\text{small}}(x_{t_{i-1}},t_{i-1},t_i)$
\State \textbf{Generate regular grid:} $\{t_0, t_1, \dots, t_n\} \gets \text{linspace}(0, T, n)$
\State \textbf{Generate jump data:} Use \ref{j-time} and \ref{j-size} to generate positive and negative jumps.
\State \textbf{Merge grids:} $ \mathcal{T} \gets \{t_0, t_1, \dots, t_n\} \cup \{\text{positive jump times}\} \cup \{\text{negative jump times}\} $
\State Sort $\mathcal{T}$ in increasing order.
\State For each $t \in \{t_0, t_1, \dots, t_n\} \subset \mathcal{T}$, set the corresponding jump size $z=0$.
\State \textbf{Initialize:} Set $X_{t_0} \gets x_0$.
\For{$i = 1$ to $|\mathcal{T}|-1$}
    \State $\Delta t \gets t_i - t_{i-1}$
    \State \textbf{Drift increment:}
    $
    \Delta X_{\text{drift}} \gets a\bigl(t_{i-1},X_{t_{i-1}}\bigr) \Delta t
    $
    \State \textbf{Diffusion increment:}
    $
    \Delta X_{\text{diff}} \gets b\bigl(t_{i-1},X_{t_{i-1}}\bigr) \sqrt{\Delta t}\, Z,\quad Z\sim \mathcal{N}(0,1)
    $
    \State \textbf{Large-jump increment:}
    \[
    \Delta X_{\text{jump}} \gets
    \begin{cases}
      c\bigl(t_{i-1},X_{t_{i-1}}, z_{t_{i-1}}\bigr), & \text{if a jump occurs at } t_i,\\[1mm]
      0, & \text{otherwise}.
    \end{cases}
    \]
    \State \textbf{Small-jump increment:}
    \[
    \Delta X_{\text{small}} \gets \sqrt{v} Z', \; Z'\sim \mathcal{N}(0,1), \;  v \gets \sigma^2_{\text{small}}\bigl(X_{t_{i-1}}, t_{i-1}, t_i\bigr).
    \]
    \State \textbf{Update:}
    $
    X_{t_i} \gets X_{t_{i-1}} + \Delta X_{\text{drift}} + \Delta X_{\text{diff}} + \Delta X_{\text{jump}} + \Delta X_{\text{small}}
    $
\EndFor
\State \Return $\{X_{t_i}\}_{i=0}^{|\mathcal{T}|-1}$
\end{algorithmic}
\end{algorithm}

\section{Numerical Examples and Comparisons}\label{Sim}

We compare  two approaches: AR -- the one with a fixed cutting level $\eps$, and DC -- a new dynamic cutting method described in this paper. In this section, we illustrate the performance of these two cutting schemes. In order to  compare our results with those obtained in  \cite{BM24}, we choose the same SDE:
\[
\begin{cases}
    dX_t = \sin\bigl(X_t\bigr)\,dt +\int_{\Real}\cos(X_{t-})\,z \,\tN(dt,dz),
    \\[6pt]
    X_0 = 0, \quad t \in [0, 1],
\end{cases}
\]
where $\tN(dt,dz)$ is the compensated Poisson random measure, corresponding to a L\'evy process with L\'evy measure
\[
\nu(dz) = \mathbf{1}_{\{|z|\leq 1\}} \frac{dz}{|z|^{1+\alpha}}.
\]

We generate a “benchmark” path of $2^{17}$ points; then, for coarser grids ranging from $2^9$ to $2^{16}$ points, we measure the \emph{strong errors} in $L^p(\Omega)$-norm, for $p=2,4,6,8,10$:
\[
\Bigl\|\sup_{0\le t\le 1} |X_t^{\mathrm{benchmark}} - X_t^{\mathrm{coarse}}|\Bigr\|_{L^p(\Omega)}.
\]

\textbf{DC Construction}. We have
\[
    N^+(r) = \int_r^\infty \nu(dz) = \frac{-1 + r^{-\alpha}}{\alpha}, \quad r < 1,
\]
\[
    \tau^+(t) = \Big\{ r: N^+(r) = \frac{1}{t} \Big\} = \left( \frac{\alpha + t}{t} \right)^{-1/\alpha}, \quad t > 0,
\]
\begin{align*}
    \sigma^2_{\text{small}}\bigl(X_{t_{i-1}}, t_{i-1}, t_i\bigr)
    &= 2 \int_{t_{i-1}}^{t_i} \!\int_0^{\tau((sh)^\eps)} \left(\cos(X_{t_{i-1}}) \, z \right)^2 \, \nu(dz) \, ds \\
    &\approx 2 \cos^2\left(X_{t_{i-1}}\right) (t_i - t_{i-1}) \frac{\tau\Bigl((t_{i-1} h)^\eps\Bigr)^{2-\alpha}}{2-\alpha}.
\end{align*}

\textbf{Simulation Parameters}:
\begin{itemize}
    \item \textbf{Benchmark grid:} $2^{17}$ steps.
    \item \textbf{Coarse grids:} $2^k$ steps for $k=9,\dots,16.$
    \item \textbf{Monte Carlo loops:} 100.
    \item \textbf{Number of trajectories per loop:} 1000.
\end{itemize}

We perform 3 groups of simulations: 1) $\alpha=0.5$, 2) $\alpha=1.0$, 3) $\alpha=1.5$. For each of them we choose $\eps_{AR}=0.01$. In the DC scheme, we choose $\eps_{DC}=0.1$; $h$ is chosen in such a way that the small-jump variance is the same as in the AR method.
\begin{itemize}
    \item \textbf{$\alpha=0.5$:} $\eps_{AR} = 0.01, \; \eps_{DC} = 0.1, \; h = 6.810 \times 10^{-13}$
    \item \textbf{$\alpha=1.0$:} $\eps_{AR} = 0.01, \; \eps_{DC} = 0.1, \; h = 2.877 \times 10^{-20}$
    \item \textbf{$\alpha=1.5$:} $\eps_{AR} = 0.01, \; \eps_{DC} = 0.1, \; h = 1.575 \times 10^{-28}$
\end{itemize}

Below, we have the calculated strong errors and log-scaled plots for comparison.

\begin{figure}[H]
    \centering
    \includegraphics[width=0.65\textwidth]{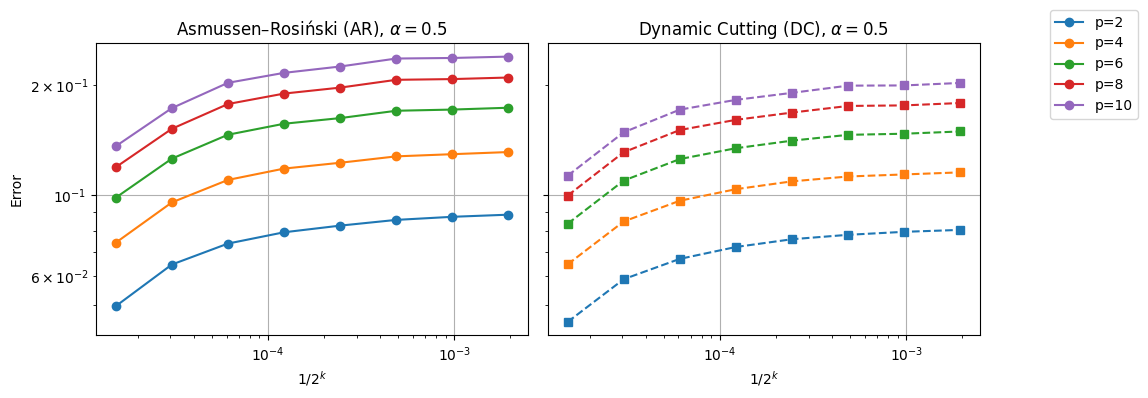}
    \caption{Strong errors for $\alpha = 0.5$ using AR and DC methods.}
    \label{alpha_05}
\end{figure}

\begin{figure}[H]
    \centering
    \includegraphics[width=0.65\textwidth]{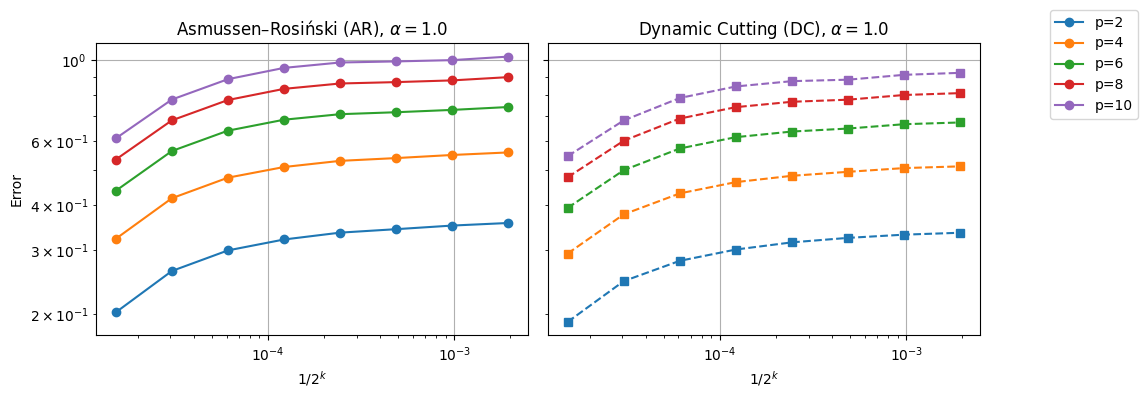}
    \caption{Strong errors for $\alpha = 1.0$ using AR and DC methods.}
    \label{alpha_10}
\end{figure}

\begin{figure}[H]
    \centering
    \includegraphics[width=0.65\textwidth]{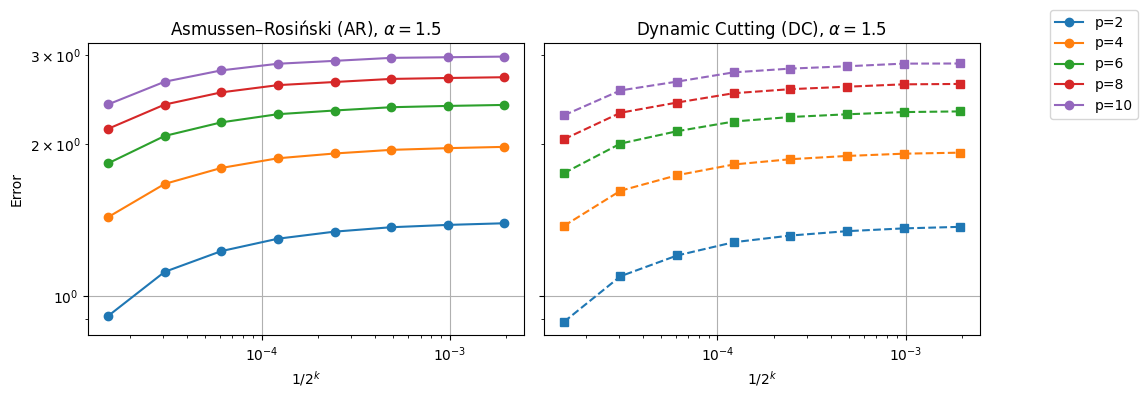}
    \caption{Strong errors for $\alpha = 1.5$ using AR and DC methods.}
    \label{alpha_15}
\end{figure}

\begin{table}[H]
    \centering
    \caption{AR--DC values for varying $k$ (rows) and $p$ (columns). Each cell reports values for $\alpha = 0.5$ (top), $\alpha = 1.0$ (middle), and $\alpha = 1.5$ (bottom).}\label{Tab1}
    \begin{tabular}{|c|c|c|c|c|c|}
        \hline
        \backslashbox{$k$}{$p$} & 2    & 4    & 6    & 8    & 10   \\
        \hline
        16  & \makecell{0.004697 \\ 0.011908 \\ 0.024332}
        & \makecell{0.009440 \\ 0.028674 \\ 0.056420}
        & \makecell{0.014893 \\ 0.044859 \\ 0.080354}
        & \makecell{0.019510 \\ 0.056620 \\ 0.099960}
        & \makecell{0.023049 \\ 0.064859 \\ 0.116611} \\
        \hline
        15  & \makecell{0.005692 \\ 0.016352 \\ 0.022110}
        & \makecell{0.010794 \\ 0.040496 \\ 0.051140}
        & \makecell{0.016229 \\ 0.064247 \\ 0.071091}
        & \makecell{0.020813 \\ 0.082767 \\ 0.087895}
        & \makecell{0.024397 \\ 0.096565 \\ 0.103538} \\
        \hline
        14  & \makecell{0.006770 \\ 0.019122 \\ 0.024264}
        & \makecell{0.013514 \\ 0.044997 \\ 0.058592}
        & \makecell{0.020774 \\ 0.068273 \\ 0.087603}
        & \makecell{0.026792 \\ 0.086078 \\ 0.115085}
        & \makecell{0.031517 \\ 0.099598 \\ 0.140047} \\
        \hline
        13  & \makecell{0.007003 \\ 0.019759 \\ 0.020785}
        & \makecell{0.014376 \\ 0.046322 \\ 0.052462}
        & \makecell{0.022395 \\ 0.071730 \\ 0.075594}
        & \makecell{0.028880 \\ 0.091236 \\ 0.094709}
        & \makecell{0.033767 \\ 0.105668 \\ 0.111269} \\
        \hline
        12  & \makecell{0.006773 \\ 0.019989 \\ 0.023732}
        & \makecell{0.013534 \\ 0.047694 \\ 0.049471}
        & \makecell{0.021497 \\ 0.073909 \\ 0.068252}
        & \makecell{0.028612 \\ 0.094288 \\ 0.085567}
        & \makecell{0.034333 \\ 0.109841 \\ 0.101892} \\
        \hline
        11  & \makecell{0.007666 \\ 0.018371 \\ 0.024347}
        & \makecell{0.015193 \\ 0.044797 \\ 0.053486}
        & \makecell{0.023921 \\ 0.070594 \\ 0.075387}
        & \makecell{0.031346 \\ 0.091726 \\ 0.094267}
        & \makecell{0.037093 \\ 0.108467 \\ 0.111581} \\
        \hline
        10  & \makecell{0.007872 \\ 0.019638 \\ 0.021888}
        & \makecell{0.015492 \\ 0.043695 \\ 0.049074}
        & \makecell{0.024168 \\ 0.063193 \\ 0.066032}
        & \makecell{0.031694 \\ 0.077715 \\ 0.077300}
        & \makecell{0.037656 \\ 0.088903 \\ 0.085588} \\
        \hline
        9  & \makecell{0.008080 \\ 0.021382 \\ 0.022657}
        & \makecell{0.015750 \\ 0.047120 \\ 0.051622}
        & \makecell{0.024048 \\ 0.069267 \\ 0.069708}
        & \makecell{0.031078 \\ 0.086520 \\ 0.082181}
        & \makecell{0.036547 \\ 0.099985 \\ 0.092554} \\
        \hline
    \end{tabular}
\end{table}

\appendix

\section{Auxiliary Lemmas}

\begin{lemma}\label{int_c_p_1}
    For any $p \ge 2$,
    \[
        \int_{B(s,h,\eps)} \left| c(s,x,z) \right|^p \, \nu(dz)
        \lesssim (1 + |x|)^p s^{-\sigma}  \frac{\big(\tau((sh)^\eps)\big)^p}{(sh)^\eps},
    \]
   where $B(s,h,\eps)$  and $\tau(t)$ are defined in  \eqref{Bsh}, and \eqref{tau}, respectively.
\end{lemma}

\begin{proof}
    Define $\xi := \frac{1}{\tau((sh)^\eps)}$. For $p \ge 2$, by Assumption \eqref{NC} and definition \eqref{Set15}, we have
    \begin{align*}
        \int_{B(s,h,\eps)} \left| c(s,x,z) \right|^p \, \nu(dz)
        &\lesssim  s^{-\sigma}  \int_{B(s,h,\eps)} (1 + |x|)^p |z|^p \, \nu(dz) \\
        &= (1 + |x|)^p s^{-\sigma}  \int_{|z| \le \tau((sh)^\eps)} |z|^p \, \nu(dz) \\
        &= \frac{(1 + |x|)^p}{\xi^p} s^{-\sigma}   \int_{|z\xi| \le 1} |z\xi|^p \, \nu(dz) \\
        &\le \frac{(1 + |x|)^p}{\xi^p} s^{-\sigma}  \int_{|z\xi| \le 1} |z\xi|^2 \, \nu(dz) \\
        &\lesssim \frac{(1 + |x|)^p}{\xi^p} s^{-\sigma}  \psi^L(\xi),
    \end{align*}
    where we used that$|z\xi|^p \le |z\xi|^2$ since $|z\xi| \le 1$.

Since   $N^\pm(\xi) \asymp \psi^{L,\pm}(1/\xi)$ (cf. Assumpion~\ref{N}), we  obtain
    \begin{align*}
        \int_{B(s,h,\eps)} \left| c(s,x,z) \right|^p \, \nu(dz)
        &\lesssim \frac{(1 + |x|)^p s^{-\sigma} }{\xi^p} N(1/\xi) \\
        &= (1 + |x|)^p s^{-\sigma}  \big(\tau((sh)^\eps)\big)^p N\big(\tau((sh)^\eps)\big) \\
        &= (1 + |x|)^p s^{-\sigma}   \frac{\big(\tau((sh)^\eps)\big)^p}{(sh)^\eps},
    \end{align*}
   where $ \psi^L(\xi)$ is defined in  \eqref{Set15}.
\end{proof}

\begin{lemma}\label{int_c_p_2}
Let $\alpha\in (0,2)$ be from \eqref{alpha}.
    Then, for any $p > \alpha$, we have
    \[
        \int_{B(s,h,\eps)} |c(s,x,z)|^p \, \nu(dz)
        \lesssim (1 + |x|)^p s^{-\sigma}   \big(\tau((sh)^\eps)\big)^{p-\alpha}.
    \]
    where $B(s,h,\eps)$  and $\tau$ are defined in \eqref{Bsh} and \eqref{tau}, respectively.
\end{lemma}

\begin{proof}
    Consider the inclusion
    \[
        B(s,h,\eps) \subset \widetilde{B}(s,h,\eps) := \big\{ z: -\tau((sh)^\eps) < z < \tau((sh)^\eps) \big\}.
    \]
    Using Assumption \ref{NC} and applying Fubini–Tonelli theorem, we obtain
    \begin{align*}
        \int_{B(s,h,\eps)} |c(s,x,z)|^p \, \nu(dz)
        &\lesssim (1 + |x|)^p  s^{-\sigma} \int_{B(s,h,\eps)} |z|^p \, \nu(dz) \\
        &\leq (1 + |x|)^p s^{-\sigma}  \int_{\widetilde{B}(s,h,\eps)} |z|^p \, \nu(dz) \\
        &= (1 + |x|)^p s^{-\sigma}  \int_0^{\tau((sh)^\eps)} p \, t^{p-1} \, \nu\big(\{ z\in \real:  t < |z| \le \tau((sh)^\eps) \}\big) \, dt \\
        &\leq (1 + |x|)^p s^{-\sigma}  \int_0^{\tau((sh)^\eps)} p \, t^{p-1} \, \frac{1}{t^\alpha} \, dt \\
        &= (1 + |x|)^p s^{-\sigma}  \frac{p}{p-\alpha} \big( \tau((sh)^\eps) \big)^{p-\alpha}.\qedhere
    \end{align*}
\end{proof}

\begin{lemma}\label{E_c_p_1}
    Let $X_t$ be a L\'evy-type process.
    Then for any $p \ge 2$
    \[
        \E{\int_0^t \!\! \int_{B(s,h,\eps)} \big| c(s,X_s,z) \big|^p \, \nu(dz) \, ds}
        \lesssim \frac{\big(\tau(h^\eps)\big)^p}{h^\eps},\qedhere
    \]
    where $B(s,h,\eps)$  and $\tau(s)$ are defined in \eqref{Bsh} and \eqref{tau}, respectively.
\end{lemma}

\begin{proof}
    We use the result obtained in Lemma \ref{int_c_p_1}, Corollary~\ref{CorB7} to  Theorem \ref{tau_scale} and
    \cite[Th.4.1.]{K17}, where it is shown that  $\E{\sup_{s \in [0,t]} |X_s|^p} \lesssim 1$ if $X_0$ belongs to some  compact set $K$,
    \begin{align*}
        \E{\int_0^t \!\! \int_{B(s,h,\eps)} \big| c(s,X_s,z) \big|^p \, \nu(dz) \, ds}
        &\lesssim \E{\int_0^t \big(1 + |X_s|\big)^p s^{-\sigma}  \frac{\big(\tau((sh)^\eps)\big)^p}{(sh)^\eps} \, ds} \\
        &\lesssim \frac{\big(\tau(h^\eps)\big)^p}{h^\eps} \, \E{\int_0^t \big(1 + |X_s|\big)^p s^{-\sigma-\eps} \, ds} \\
        &\lesssim \frac{\big(\tau(h^\eps)\big)^p}{h^\eps} \, \E{\sup_{s \in [0,t]} \big(1 + |X_s|\big)^p} \int_0^t s^{-\sigma-\eps} \, ds \\
        &\lesssim \frac{\big(\tau(h^\eps)\big)^p}{(th)^\eps} \, \frac{t^{1-\sigma}}{1-\eps-\sigma} \\
        &\lesssim \frac{\big(\tau(h^\eps)\big)^p}{h^\eps}.\qedhere
    \end{align*}
\end{proof}

\begin{lemma}\label{E_c_p_2}
   Let $\alpha\in (0,2)$ be from \eqref{alpha}.
    Then for any $p > \max(\alpha, 1)$
    \[
        \E{\int_0^t \!\! \int_{B(s,h,\eps)} \big| c(s,X_s,z) \big|^p \, \nu(dz) \, ds}
        \lesssim  \big( \tau(h^\eps) \big)^{p-\alpha},
    \]
    where $B(s,h,\eps)$  and $\tau(s)$ are  defined in \eqref{Bsh} and \eqref{tau}, respectively.
\end{lemma}

\begin{proof}
    We use the result obtained in Lemma \ref{int_c_p_2}, the corollary from Theorem \ref{tau_scale} and that $\E{\sup_{s \in [0,t]} |X_s|^p} \lesssim 1$ (cf. \cite[Th.4.1.]{K17} ):
    \begin{align*}
        \E{\int_0^t \!\! \int_{B(s,h,\eps)} \big| c(s,X_s,z) \big|^p \, \nu(dz) \, ds}
        &\lesssim \frac{p}{p-\alpha} \big( \tau(h^\eps) \big)^{p-\alpha} \E{\int_0^t \big(1 + |X_s|\big)^p s^{-\sigma} \, ds} \\
        &\lesssim  \big( \tau(h^\eps) \big)^{p-\alpha} \int_0^t \sup_{s \in [0,t]} \E{\big(1 + |X_s|\big)^p}  s^{-\sigma} \, ds \\
        &\lesssim \big( \tau(h^\eps) \big)^{p-\alpha}.\qedhere
    \end{align*}
\end{proof}

\section{Auxiliary Theorems}

\begin{theorem}\cite[Cor.IV.4.2]{RY99}, \cite[Th.4.14]{KS23}\label{bdg} (Burkholder-Davis-Gundy Inequality)

 Let $ T $ be any stopping time and $ p \in (0, \infty) $. Then there exists a constant $ C_p^{\mathrm{BDG}} > 0 $, depending only on $ p $, such that for any continuous local martingale $ M $, the following inequality holds:
\[
\E{\sup_{0 \leq t \leq T} |M_t|^p} \leq C_p^{\mathrm{BDG}} \, \E{\langle M \rangle_T^{p/2}},
\]
where $ \langle M \rangle_T $ denotes the quadratic variation of $ M $ up to time $ T $.
\end{theorem}

 Recall (cf. \cite{Sch16}) that a predictable sigma-algebra $\mathcal{P}$ is the  smallest sigma-algebra on $(0,\infty)\times \Omega$,
such that all left-continuous adapted stochastic processes $(t,\omega)\mapsto X_t(\omega)$   are measurable.  A stochastic process $F(t,y)$  is called predictable, if it is $\mathcal{P}$-- measurable.
\begin{theorem}\cite[Th.4.20]{KS23}\label{Novikov}
    Let $(X_t)_{t \geq 0}$ be a one-dimensional stochastic process given by
    \[
    X_t = \int_0^t \int_{y \neq 0} F(s,y) \tN(dy,ds), \quad t \geq 0,
    \]
    where $ F $ is a predictable stochastic process, and $\tN = N - \widehat{N}$ is a compensated Poisson random measure with compensator $\widehat{N}(dy,ds) = \nu(dy)ds$,
    where the measure $ \nu $ satisfies
     $\int_{y \neq 0} \min\{1, |y|^2\} \nu(dy) < \infty$.
    \begin{enumerate}
        \item  \cite{No75} If
        \[
        \mathbb{E} \left[ \int_0^T \int_{y \neq 0} |F(s,y)| \nu(dy) ds \right] < \infty,
        \]
        for some $ T > 0 $, then for any $ a \in [1,2] $ and $ p \in [0, a] \subseteq [0,2] $, we have
        \begin{equation}\label{Novikov1}
        \mathbb{E} \left( \sup_{t \leq T} |X_t|^p \right)
        \leq C_{p,a} \mathbb{E} \left[ \left( \int_0^T \int_{y \neq 0} |F(s,y)|^a \nu(dy) ds \right)^{p/a} \right].
       \end{equation}

        \item  \cite[Th.2.11]{Ku04}
        For all $ p \geq 2 $, we have
        \begin{equation}\label{Novikov2}
            \begin{split}
        \mathbb{E} \left( \sup_{t \leq T} |X_t|^p \right)
        &\leq c_p \mathbb{E} \left[ \left( \int_0^T \int_{y \neq 0} |F(s,y)|^2 \nu(dy) ds \right)^{p/2} \right] \\
        &\quad + c_p \mathbb{E} \left[ \int_0^T \int_{y \neq 0} |F(s,y)|^p \nu(dy) ds \right].
        \end{split}
        \end{equation}
    \end{enumerate}
\end{theorem}

\begin{theorem}[Martingale Version of the Esseen's Theorem {\cite[Th.1]{HJ88}}] \label{HJ88}
Let $ (Y_j)_{j=1}^n $ be a sequence of square-integrable martingale differences adapted to the filtration $ (\Ff_j)_{j=0}^n $,
\[
\E{Y_j} = 0, \quad \E{Y_j^2} = \sigma_j^2 < \infty, \quad
\E{Y_j \mid \Ff_{j-1}} = 0, \quad \text{and} \quad \sum_{j=1}^n \sigma_j^2 = 1.
\]

Then, for all $ x \in \Real $, the following inequality holds:
\[
\big| \Prob{S_n < x} - \Phi(x) \big|
\leq \frac{C_\delta \, L_{n, 2\delta}^{\frac{1}{3 + 2\delta}}}{1 + |x|^{2 + 2\delta}},
\]
\begin{itemize}
    \item $\Phi(x)$ is the cumulative distribution function of the standard normal distribution,
    \item $ C_\delta > 0 $ is a constant depending only on $\delta > 0$,
    \item $ S_n := \sum_{j=1}^n Y_j $,
    \item
    $
        L_{n, 2\delta} = \sum_{j=1}^n \E{|Y_j|^{2+2\delta}} \leq 1.
    $
\end{itemize}
\end{theorem}

\begin{theorem}[Integral Inequality for Wasserstein Distance {\cite[Cor.3.2]{B18}}] \label{bobkov_wp_bound}
Let $\mu$ and $\nu$ be probability measures belonging to the space $\mathcal{M}_p$,  of measures with finite $p$-th moments, $p\geq 1$. Suppose that $F$ and $G$ are the cumulative distribution functions  corresponding to $\mu$ and $\nu$, respectively. Then, the $p$-Wasserstein distance  $\Ww_p(\mu, \nu)$ between $\mu$ and $\nu$  satisfies the inequality:
\[
\Ww_p(\mu, \nu) \leq \int_{-\infty}^\infty \lvert F(x) - G(x) \rvert^{1/p} \, dx.
\]
\end{theorem}

\begin{theorem}[Extended Gronwall Lemma {\cite[Lem.A.1]{BM24}}] \label{gronwall}
Let $f, \ell : [0, T] \to \Real_+$ be non-decreasing functions and $g, h, k \in L^1\left([0, T], \Real_+\right)$. If for some $p > 1, \; \forall t \in [0, T]$
\[
    f(t) \leq \int_0^t g(s) f(s) \, ds
    + \left( \int_0^t h(s) f^2(s) \, ds \right)^{\frac{1}{2}}
    + \left( \int_0^t k(s) f^p(s) \, ds \right)^{\frac{1}{p}}
    + \ell(t),
\]
then
\[
    f(t) \leq 2 \, p \, \ell(t)
    \exp \left( \int_0^t \big( 2p g(s) + 2p h(s) + k(s) \big) \, ds \right).
\]
\end{theorem}

Finally, we have the time scaling of  $\tau(t)$.
\begin{theorem}\cite[Lem.~4.3]{IKP25}\label{tau_scale}
    For $ R > 1 $, there exists a constant $ \zeta = \zeta(R) > 0 $ such that
    \[
        \tau(Rt) \leq R^{\zeta} \tau(t) \quad \text{for all } t > 0.
    \]
\end{theorem}

\begin{corollary}\label{CorB7}
    Since $\tau(t), \, t \in [0, T]$, is a monotonic non-decreasing function,
    for any $h,\eps\in (0,1)$
    \[
        \tau((th)^\eps) \le \tau(T^\eps h^\eps) \le \tau(R h^\eps) \le R^{\zeta} \tau(h^\eps) \lesssim \tau(h^\eps),
    \]
    where $R = \max(2, T^\eps)$.
\end{corollary}

\printbibliography[title={References}]

\end{document}